\documentclass[reqno,11pt]{amsart}
\usepackage{amsmath, amsfonts, amssymb}
\usepackage{tikz}
\usetikzlibrary{calc,decorations.markings}
\usepackage{graphicx}

\usepackage{amsmath, amsfonts, amssymb, dsfont, dsfont, tikz}
\usetikzlibrary{calc,decorations.markings, arrows}
\usepackage{graphicx}
\usepackage[symbol]{footmisc}

\title[Variance of $\mathcal{B}$--free integers in short intervals]{Variance of $\mathcal{B}$--free integers in short intervals}

\author[M. Avdeeva]{Maria Avdeeva}
\address[M. Avdeeva]{Mathematics Department, Princeton University, Princeton, NJ 08544, USA}
\email{mavdeeva@math.princeton.edu}

\def \B {\mathcal{B}}
\def\P {\mathbb{P}}
\def\Q {\mathbb{Q}}
\def\E {\mathbb{E}}

\def\R {\mathbb{R}}
\def\Z{\mathbb{Z}}
\def\N {\mathbb{N}}

\def\Var {\operatorname{Var}}
\def \ok {\mathcal{O}_K}
\def \a {\mathfrak{a}}
\def \p {\mathfrak{p}}

\def\mod{\operatorname{mod}}

\begin{document}
\newtheorem{theorem}{Theorem}
\newtheorem*{corollary}{Corollary}
\newtheorem{lemma}{Lemma}
\newtheorem*{proposition}{Proposition}
\theoremstyle{definition}
\newtheorem{definition}[theorem]{Definition}
\newtheorem*{remark}{Remark}
\numberwithin{equation}{section}
\begin{abstract}
We prove some new statements on the distribution of $\B$--free numbers in short intervals. In particular, we show an asymptotic result for the variance of the number of $\B$--free integers in random short intervals which are, in some sense, uniformly distributed. 

We establish a connection between our work and the paper by El Abdalaoui, Lema\'nczyk \& de la Rue on a flow associated to $\B$--free integers. In addition, we study an analog of our variance for $k$--free integers in a number field which provides new information for the corresponding dynamical system constructed  by Cellarosi \& Vinogradov.

\end{abstract}

\maketitle

\section{Introduction}
Let $\mathcal{B} = \{b_1, b_2, b_3,..\}$ be a sequence of integers greater than $1$ satisfying the following assumptions:

\begin{enumerate}
\item[(A1)] $b_i$ and $b_j$ are relatively prime for $i \neq j$;
\item[(A2)] $\sum_i 1/{b_i} < \infty.$
 \end{enumerate}
Throughout the text, we will always assume (A1) and (A2), sometimes without explicitly stating it.

We call a positive integer \textbf{\textit{$\B$--free}} if it is not divisible by any element of $\mathcal{B}$. 

 $\B$--free numbers were introduced by Erd\"os in \cite{E} as a generalization of square--free numbers. Square--free numbers correspond to $\B=\{p_1^2, p_2^2, ...\}$ where $\{p_1, p_2, ...\}$ is the set of all primes. They are essential due to their connection with the M\"obius function (some background on this subject can be found in Section \ref{dyn}). Erd\"os conjectured that for any constant $\theta>0$, for all sufficiently large $x$, the interval $\left[x, x+x^{\theta}\right)$ must contain a $\B$--free number, and proved existence of a $\theta<1$ which is independent of $\B$ and has this property. Since then, $\B$--free integers in short intervals and, consequently, the gaps between them have been extensively studied (e.g., see \cite{Kow, Mat, P, Wu}). 
 
 In this paper, we analyze the average behavior of $\B$--free integers in the intervals that are the ``shortest possible,'' namely, in the family of intervals $\{[x, x+N),\, 1\leq x \leq X\}$, as $X \to\infty$, for $N$ independent of $x$. Let us clarify what this means.
 
Denote by $\mu^{\B}$ the indicator of the set of $\B$--free integers:
\begin{equation*}
\mu^{\mathcal{B}} (n) = 
\begin{cases}
1, &\text{if } n\text{ is }\B\text{--free}; \\
 0, &\text{otherwise}.\\
\end{cases}
\end{equation*}

For $\B=\{p_1^k, p_2^k, ....\}$, $\B$--free numbers are called \textit{$k$--free}. In \cite{Mir}, Mirsky studied the distribution of patterns of $k$--free integers and computed their correlation functions, i.e., the limits of correlations of their indicator. 

One can show that, under the assumptions (A1) and (A2), all the frequencies of patterns of $\mu^{\mathcal{B}}$ also exist. More precisely, for any $r\geq 0$, for any finite set $0 \leq h_1 \leq h_2 \leq ... \leq h_r$ (all integers), the limits
\begin{equation*}
c^{\B}_{r+1} (h_1, h_2\ldots h_r) := \lim_{X\to\infty} \frac 1 X  \sum_{1 \leq n \leq X} \mu^{\B} (n) \, \mu^{\B} (n+h_1) \ldots \mu^{\B} (n+h_r)
\end{equation*}
exist. We call them \textit{$(r+1)$--st correlation functions} of ${\B}$--free integers.

The formulae for these correlation functions are given in Lemma \ref{BM}, Section \ref{dyn}. Those readers who are interested in a dynamical proof can see \cite{L}, Theorem 4.1.

Lemma \ref{BM} allows the sequence $\{h_i\}$ to be empty. In this case, it gives the density $\rho^{\B}$ of the $\B$--free numbers and we will see that
\begin{equation}\label{brho}
\rho^{\B} = \prod_{b\in\B} \left(1-\frac 1 {b}\right).
\end{equation}

For example, for square--free integers, we regain their well--known density $$\rho^{\,\{p^2\}} = \prod\limits_p \left(1-\frac 1 {p^2}\right) = \frac 1 {\zeta(2)} = \frac 6 {\pi^2}.$$ 




Denote the uniform distribution on $\{1\ldots X\}$ by $\P_X$ and fix a positive integer $N$. The equation (\ref{brho}) implies that the expected number of $\B$--free integers in a random interval $[x, x+N)$ when $x$ is distributed according to $\P_X$ asymptotically equals $\rho^{\B} N$, as $X \to \infty$. Moreover, the existence of the second correlation functions guarantees that the variance of this random variable exists, i.e., one can define the central object of this paper which we call the \textbf{\textit{variance of $\B$--free numbers in short intervals}}:
\begin{align} \label{var}
\Var^{\B} (N) :&= \lim_{X\to\infty} \frac 1 X \sum_{1\leq x \leq X} \left(\sum_{x\leq n < x+N} \mu^{\B} (n) - \rho^{\B} N\right)^2 \notag \\&= \lim_{X\to\infty} \E_{\,\P_X} \left(\sum_{x\leq n < x+N} \mu^{\B} (n) - \rho^{\B} N\right)^2.
\end{align}

Let us fix the notation. In this text, we will use the following standard notation for two arithmetic functions, $f(n)$ and $g(n)$:
\begin{itemize}
\item $f(n) = O(g(n))$, or, equivalently, $f(n)\lesssim g(n)$, if there exists a positive constant $C$ and some $n_0$ such that $f(n)\leq C \,g(n)$ for all $n\geq n_0$;
\item $f(n) =o(g(n))$ if $f(n)/g(n) \to 0$, as $n\to\infty$;
\item $f(n) \sim g(n)$ if $f(n)/g(n) \to 1$, as $n\to\infty$;
\item $f(n) = \omega(g(n))$ if $f(n)/g(n) \to \infty$, as $n \to \infty$.
\end{itemize} 
\vspace{0.5cm}

We prove the following asymptotic statement for the variance.

\begin{theorem}\label{bvar}
Let $\mathcal{Z}^{\B}$ be the set of \textbf{$\B$--integers}, i.e., the multiplicative semigroup generated by $\B$.

Suppose that $2 \not \in \B$ and the counting function 
\begin{equation}\label{count}
\mathcal{N}^{\B} (N) := \#\{n \in \mathcal{Z}^{\B}, 1\leq n \leq N\}
\end{equation}
satisfies $\mathcal{N}^{\B} (N) = A \,N^{\alpha} + O\left(N^{\beta}\right)$, as $N \to \infty$, for some constants $A>0$, $0 \leq \beta <\alpha <1$. Then
\begin{equation*}
\Var^{\B} (N) \sim C \, N^{\alpha} \tag*{($N\to\infty$)}
\end{equation*}
where $C$ is independent of $N$ and equals $$\frac{2^{\alpha} {\pi}^{\alpha-2} A} {\alpha-1} \prod_{b\in \B} \left(1-\frac {2} b + \frac 2 {b^{\alpha+1}} -\frac 1 {b^{2\alpha}}\right) \sin\left(\frac {\pi(\alpha-1)} 2\right) \Gamma(2-\alpha) \zeta(2-\alpha).$$
\end{theorem}

\begin{remark} With a little additional work along the lines of our proof, it is possible to obtain an $O_{\B}\left(N^{\alpha} e^{-\operatorname{const} \sqrt{N}}\right)$ error term for the variance.\footnote[2]{As usually, here the subscript in the $O$--notation emphasizes dependence on $\B$.}


\end{remark}

Now, for each $X$ and $N$, consider the probability of the so-called \textit{exceptional set}, i.e., the set of the intervals containing no $\B$--free numbers: $$P_{\B}(X, N) := \P_X \{x\,|\,1\leq x \leq X, \sum_{x\leq n < x+N} \mu^{\B} (n) = 0\}.$$

Combining existence of asymptotic frequencies for $\B$--free integers (Lemma \ref{BM}) and Theorem \ref{bvar} with Markov's inequality, we immediately deduce
\begin{corollary}\label{cor}
Under the assumptions of Theorem \ref{bvar}, for any integer $N\geq 1$, $\lim_{X\to\infty} P_{\B}(X,N)$ exists and, as $N \to \infty$,
\begin{equation*}
\lim_{X\to\infty} P_{\B}(X,N) \lesssim N^{\alpha-2}.
\end{equation*}
\end{corollary}

This corollary can be compared with the results of Matom\"aki \cite{Mat}.
In \cite{Mat}, $N$ was allowed to grow with $X$ and, provided that $N \leq X^{\frac 1 6 -\delta}$ for some $\delta>0$, Matom\"aki showed that $P_{\B} (X,N) \lesssim N^{-\theta}$ for all $X$ and any $\theta<1$. Therefore, although our assumptions on $N$ are more restrictive, we prove a stronger statement in the case considered.

Theorem \ref{bvar} will be proven in Section \ref{th1proof} using the methods of complex analysis and some technical lemmata that are contained in Section \ref{tech}. Prior to that, in Section \ref{dyn}, we will briefly discuss the flows associated to square--free numbers that were considered by Sarnak \cite{Sar} and Cellarosi and Sinai \cite{CS} and the generalization of these flows to $\B$--free integers (introduced by El Abdalaoui, Lema\'nczyk and de la Rue \cite{L}). Our variance can be defined entirely in number--theoretical terms and, as we will see, the dynamical approach is unnecessary for the proofs of our results. Nevertheless, we prefer the dynamical perspective and view the variance as a characteristic of the corresponding dynamical system.
\\

Omitting a finite number of the elements of the sequence $\B$ affects only the constant $C$.

One should mention that the case of $k$--free integers was resolved previously. Namely, for $\B = \{p_1^k, p_2^k, \ldots\}$, the function $\mathcal{N}^{\B} (N)$ defined in (\ref{count}) simply counts the number of integer $k$--th powers up to $N$, thus, $\alpha=\frac 1 k$ and $\beta = 0$. Therefore, by Theorem \ref{bvar}, the variance of $k$--free numbers in short intervals of length $N$ behaves asymptotically as $C(k)\,N^{1/k}$. This was obtained with no error term by different methods in \cite{DT}. The even more special case of square--free integers, $k=2$, was covered by Hall (see \cite{H}, Lemma 1) whose approach we use in the proof of Theorem \ref{bvar}. In a later work \cite{H2}, Hall gave a bound on the higher moments of square--free integers in short intervals. The author obtained similar bounds for the higher moments of $\B$--free integers in short intervals, however, believes that improving Hall's results is a very hard problem.
\\

A natural step to further generalization is to consider $k$--free integers in a general number field. Let $K/\Q$ be a degree $d$ extension with a ring of integers $\ok$. One way to define a $k$--free integer in this setting is to declare an element of $\ok$ $k$--free if its principal ideal is a product of distinct prime ideals. This was done by Cellarosi and Vinogradov in \cite{CV}. In Section \ref{kfree}, we discuss $k$--free integers in number fields, a corresponding flow from \cite{CV} and define for them an analog of the variance in short intervals which we denote by $\Var_{\ok}^{(k)}$. We prove the following bound on this variance.
\begin{theorem}\label{okvar}
Let the variance $\Var_{\ok}^{(k)} (N)$ of $k$--free integers in the ring of integers $\ok$ of the number field $K$ be defined by (\ref{varok}). Then
\begin{equation}\label{okbound}
\Var_{\ok}^{(k)} (N) = O_{K, k} \left(N^{2d-1}\right),
\end{equation}
as $N\to\infty$, where $d$ is the degree of $K$ over $\Q$.
\end{theorem} 

For instance, for Gaussian integers the degree of the extension equals $2$, therefore, the theorem gives an $O\left(N^3\right)$ bound.

See Section \ref{th2proof} for the proof of this result. 

\subsection*{Acknowledgements} The author thanks her advisor Yakov G. Sinai as well as Francesco Cellarosi, Maksym Radziwill and Ilya Vinogradov for helpful discussions and suggestions.

\section{The dynamics of square--free and $\B$--free integers}\label{dyn}
Recall that the M\"obius function is defined as follows:
\begin{equation*}
\mu(n)=
\begin{cases}
1, &\text{if } n=1;\\
(-1)^r, &\text{if } n \text{ is the product of }r\text{ distinct primes};\\
0, &\text{if }n \text{ is not square--free}.
\end{cases}
\end{equation*}

This function is important in Number Theory for many reasons, one of which is that $\left|\sum_{n=1}^N \mu(n)\right| = O\left(N^{\frac 1 2 +\varepsilon}\right)$ is equivalent to the Riemann Hypothesis. This is why studying randomness of $\mu$ is so appealing.

The function $\mu^2$ is the indicator of square--free integers. 
Sarnak \cite{Sar} and Cellarosi and Sinai \cite{CS} considered two different flows which reflect the structure of square--free integers. The first one is the subshift on $\{0,1\}^{\Z}$ determined by the sequence $\mu^2$ (for an idea of the construction, see the description of the generalized system $\left(Y, \nu^{\B}, T\right)$ below) while the other one is a translation on the compact abelian group $\prod_p \Z/p^2 \Z$ equipped with the Haar measure. These flows are ergodic, have zero Kolmogorov--Sinai entropy and can be shown to have the same pure point spectrum. Hence these dynamical systems are isomorphic and one can, actually, find an explicit measurable isomorphism between them.

%
Since in the present paper we are mostly concerned with $\B$--free integers, we wish to proceed to this setting. In \cite{L}, El Abdalaoui et al generalized the results of Sarnak and Cellarosi \& Sinai and built the \textit{$\B$--free flow}, a dynamical system associated to $\B$--free numbers.

We start with the asymptotic frequencies for $\mu^{\B}$. For $b \in \B$, let $$A_b (h_1, \ldots h_r)\!:= \!\#\left\{0\, \mod\, b, h_1 \,\mod\, b, \ldots h_r \, \mod \,b\right\}.$$
The following generalization of Mirsky's formulae holds (\cite{L}, Theorem 4.1):

\begin{lemma}\label{BM}
Suppose a set $\B$ satisfies (A1) and (A2). Then, for any finite sequence $0 \leq h_1 \leq h_2 \leq ... \leq h_r$, the limit $$\lim_{X\to\infty} \frac 1 X  \sum_{1 \leq n \leq X} \mu^{\B} (n) \, \mu^{\B} (n+h_1) \ldots \mu^{\B} (n+h_r)$$ exists and equals
\begin{equation*}
c^{\B}_{r+1} (h_1, h_2, \ldots h_r) := \prod_{b \in \B} \left(1-\frac{A_b(h_1, h_2\ldots, h_r)} {b}\right).
\end{equation*}
\end{lemma}


The assumptions (A1) and (A2) are essential for the flow to be defined since the existence of the correlation functions, i.e., Lemma \ref{BM}, is required. In yet other words, (A1) guarantees that the processes of sieving by different $b_i$ and $b_j$ for $i \neq j$ are asymptotically independent (for an interesting account on this matter, see \cite{Gr}); at the same time, the fact that the density of $\B$--free integers equals $\prod_b \left(1-\frac 1 b\right)$ indicates that (A2) is already necessary for $\B$--free numbers to have a density. 

The $\B$--free flow built in \cite{L} is measurably isomorphic to a subshift on $Y:= \{0,1\}^{\mathbb{Z}}$. Let us sketch the construction here.

We introduce a probability measure $\nu^{\B}$ that naturally corresponds to the set of $\B$--free integers: it is the unique measure that, for any cylinder $$C_{n_0, n_1\ldots n_{r}} := \left\{x \in Y \text { such that } x(n_0) = x(n_1) = \ldots = x(n_{r})=1\right\},$$ provides 
\begin{equation}\label{meas}
\nu^{\B} (C_{n_0, n_1 \ldots n_{r}}) = c^{\B}_{r+1} (n_1-n_0, n_2-n_0, \ldots n_r-n_0).
\end{equation}

Define $T$ to be the shift on $Y$, i.e. $Tx = x'$ with $x'(n) = x(n+1)$. It follows from \cite{L} that $(Y, \nu^{\B}, T)$ is an ergodic dynamical system with pure point spectrum. 

We let $f: Y \to \R$ be the projection onto the 0--th coordinate of $x$, i.e., set $f(x):=x(0)$. Then the sequence $\{f(T^n \xi)\}_{n \in \Z}$ for any $\xi$ with $\xi(n)=\mu^2(n),\, n\in \N$, is generic for the subshift and, by construction of the measure $\nu^{\B}$, all the correlation functions of the sequence $\{\mu^{\B}(n),\,n\in\N\}$ appear as the correlation functions of the sequence $\{f(T^n x)\}_{n \in \Z}$ with respect to $\nu^{\B}$. In particular, $\E_{\nu^{\B}} f(x) = \rho^{\B}$ and the variance $\Var^{\B} (N)$ defined in (\ref{var}) coincides with
\begin{equation}\label{newvar}
\Var_{\nu^{\B}} \sum_{i=0}^{N-1} f\left(T^i x\right) = \E_{\nu^{\B}} \left(\sum_{i=0}^{N-1} f\left(T^i x\right) - \rho^{\B} N\right)^2.
\end{equation}
Thus, from Theorem \ref{bvar} (and under the same assumptions as in this Theorem), 

\begin{equation*}
\Var_{\nu^{\B}} \sum_{i=0}^{N-1} f\left(T^i x\right) \sim C \, N^{\alpha} \tag*{($N\to\infty$),}
\end{equation*}
for some constant $C$ independent of $N$.
\section{$k$--free integers in number fields}\label{kfree} Throughout this section $k\geq 2$ will be a fixed integer. Adopting the notation from \cite{CV}, we will briefly recall the necessary definitions and setup here.

Suppose $K/\Q$ is a degree $d$ extension with the ring of integers $\mathcal{O}_K$. Due to unique factorization of ideals into prime ones, for an ideal $\mathfrak{a}$ in $\mathcal{O}_K$, one can define:
\begin{equation*}
\mu^{(k)} (\mathfrak{a}) :=
\begin{cases}
1, &\mathfrak{a} \not\subseteq \mathfrak{p}^k \text{ for every prime ideal } \mathfrak{p}, \\
0, &\text{otherwise}. 
\end{cases} 
\end{equation*}

We call an ideal $\mathfrak{a}$ in $\mathcal{O}_K$ \textbf{$k$--free} if $\mu^k(\mathfrak{a})=1$. An integer $a \in \ok$ is defined to be $k$--free if the principal ideal $(a)$ is $k$--free. For any ideal $\a \subset \ok$, we denote by $N(\a)$ the algebraic norm of this ideal.

The $k$--free integers of $\ok$ have density $\rho^{(k)}_{\ok}:=\frac 1 {\zeta_K (k)}$ where $\zeta_K (s)$ is the Dedekind zeta function of the field $K$ which, for $\operatorname{Re} s >1$, is given by:

\begin{equation*}
\zeta_K (s) = \sum_{\a} N(\a)^{-s} = \prod_{\p} \left(1-N(\p)^{-s} \right)^{-1}.
\end{equation*}

Identify $\mathbb{Z}^d \subset \R^d$ with $\ok \subset K$ via any group isomorphism $\imath: \,(\Z^d, +) \to (\ok, +)$. Let $||\cdot||_{1, \ok} = ||\cdot||$ denote the image of $L^1$--norm on $\Z^d$ under $\imath$ and $B_x \subset \ok$ be the $||\cdot||$--ball of radius $x$ centered at the origin. For an ideal $\mathfrak{d}$, we also define its \textit{diameter} (minimal size of its cell over all choices of generators): $$\operatorname{diam} \mathfrak{d} = \min_{\mathfrak{d} =  \langle v_i \rangle_{i=1}^d} \max \{||a|| \text{ for } a= \sum_j \varepsilon_{ij} v_j,\, \varepsilon_{ij} = 0 \text{ or } 1\}.$$

In \cite{CV}, the stationarity of $\mu^{(k)}$ was shown, i.e.,  the existence of the correlation functions  
\begin{equation*}
c^{(k)}_{\ok, r+1} (h_1, \ldots h_r) := \lim\limits_{X\to\infty} \frac 1 {\# B_X} \sum_{n \in B_X} \mu^{(k)} (n) \mu^{(k)} (n+h_1) \ldots \mu^{(k)} (n+h_r)
\end{equation*}
for any $r\geq0$ and $h_i \in \ok, i=1\ldots r$.

Therefore, in complete analogy to the original case $K=\Q$ (see, e.g., (\ref{meas})), one can construct the unique, up to sets of measure zero, $\ok$--subshift $\nu_{\ok}^{(k)}$ satisfying $$\nu_{\ok}^{(k)}\left\{x \in \{0,1\}^{\ok}: x(h_0) = \ldots x(h_r) = 1\right\} = c^{(k)}_{\ok, r+1} (h_1-h_0, \ldots h_r-h_0).$$

Generalizing a theorem proven by Cellarosi and Sinai in \cite{CS} for the case $K=\Q$, Cellarosi and Vinogradov showed that the subshift $\nu_{\ok}^{(k)}$ is isomorphic to an action of $\mathbb{Z}^d$ on a compact abelian group with Haar measure hence it is ergodic and has pure point spectrum.

How do we adapt the question concerning the behavior of $k$--free integers in short intervals to the $d$--dimensional case of $\ok$? Now we seek to estimate the variance of the number of $k$--free $\ok$--integers in a $||\cdot||$--ball of the fixed radius $N$ centered at a point of $\ok$ which is, in some sense, uniformly distributed. Formally, this variance $\Var_{\ok}^{(k)} (N)$ is given by:
\begin{equation}\label{varok}
\Var_{\ok}^{(k)} (N) = \lim\limits_{M\to\infty} \frac 1 {\# B_M} \sum\limits_{m \in B_M} \left(\sum\limits_{n, \, n-m \in B_N} \mu^{(k)} (n) - \rho_{\ok}^{(k)} \# B_N\right)^2.
\end{equation}

A trivial bound on this variance would be $$O_{K, k}\left(\left(\# B_N\right)^2\right) = O_{K, k}\left(N^{2d}\right),$$ as $N\to\infty$. Theorem \ref{okvar} formulated in the Introduction (see equation (\ref{okbound})) states that this can be reduced to $O_{K, k}\left(N^{2d-1}\right)$. Section \ref{th2proof} is devoted to the proof of this theorem.

\section{Technical lemmata}\label{tech}
Recall that $\mathcal{Z}_{\B}$ denotes $\B$--integers, i.e., the multiplicative semigroup generated by the set $\B$, and $\mathcal{N}^{\B} (N) = \#\{n \leq N, n \in \mathcal{Z}^{\B}\}$.
Let \textit{the zeta function corresponding to the set $\B$} be given by 
$$\zeta^{\B} (s) = \sum_{n \in \mathcal{Z}^{\B}} n^{-s}.$$

Assume $\mathcal{N}^{\B} (N) = A N^{\alpha} + O\left(N^{\beta}\right)$, as $N\to\infty$, for some constants $A$, $0\leq\beta<\alpha\leq1$. Then the following two lemmata hold.

\begin{lemma}\label{extension}
The function $\zeta^{\B} (s)$ can be meromorphically extended to the half-plane $\operatorname{Re} s>\beta$, with a simple pole at $s=\alpha$ and no other poles.
\end{lemma}

\begin{lemma}\label{zero-free}
There exist $c$ and $t_0$, s.th. for $s \in \left\{\sigma+it\, |\, \sigma>\alpha-\frac {c} {\log t}, |t| \geq t_0\right\}$ we get $$\left|\zeta^{\B} (s)\right| \leq D \log t$$ for some constant $D$.
\end{lemma}

Dependence of the constants on the set $\B$ is not relevant for us.

The proof of these statements can be found, for example, in \cite{R}. The reader should not be confused by the fact that the author of \cite{R} is concerned with so--called \textit{Beurling integers}. These integers are, by definition, the elements of the multiplicative semigroup generated by any non--decreasing sequence of real numbers (not necessarily integers with conditions (A1) and (A2) like in our case). Thus, $\mathcal{Z}_{\B}$ is a special case of a set of Beurling integers and all the properties of the corresponding zeta--function defined in \cite{R} apply to $\zeta^{\B}$. We choose to skip all the details here; for a thorough introduction to Beurling integers, one can also refer to \cite{Hil}.

Another technical lemma required for the proof of Theorem \ref{bvar} is a variant of Perron's formula that can be found in \cite{Ten}, page 134:

\begin{lemma}\label{Perr} Let $f(h)$ be an arithmetic function and $F(s)=\sum_{h=1}^{\infty} \frac {f(h)} {h^s}$ its Dirichlet series with the abscissa of convergence $\sigma_c$. Then, for any $N \geq 1$ and $\kappa>\max(0,\sigma_c)$, the following summation formula holds: $$\sum_{h=1}^{N-1} (N-h) f(h) = \frac 1 {2\pi i} \int\limits_{\kappa- i\infty}^{\kappa+i\infty} F(s) N^{s+1}\, \,\frac {d s} {s(s+1)}.$$
\end{lemma}

\section{Proof of Theorem \ref{bvar}}\label{th1proof}
First, using the more convenient dynamical representation from Section~\ref{dyn}, we rewrite the variance in terms of the second correlation functions:
\begin{align*}
	\Var_{\nu^{\B}} (N) &=\E_{\nu^{\B}} \left(\sum_{i=0}^{N-1} f\left(T^i x\right) - \rho^{\B} N\right)^2 \notag\\ \notag
	&= \E_{\nu^{\B}} \left(\sum_{i=0}^{N-1} f\left(T^i x\right)\right)^2 - \left(\rho^{\B} N\right)^2\\
	&= \sum_{0 \leq j,k < N} \left[\E_{\nu^{\B}} f\left(T^j x\right) f\left(T^k x\right)\right]  - \left(\rho^{\B} N\right)^2  \notag \\
	&= \sum_{0 \leq j,k < N} c^{\B}_{2}(|k-j|)  - \left(\rho^{\B} N\right)^2 \notag\\
	&= 2 \, \sum_{h=1}^{N-1} (N-h) c^{\B}_{2} (h) + \rho^{\B} N - \left(\rho^{\B} N\right)^2. 
\end{align*}
\newline

Denote $S^{\B} (N):= \sum_{h=1}^{N-1} (N-h) c^{\B}_{2} (h)$ so that  
\begin{equation}\label{varvar}
\Var_{\nu^{\B}} (N) =2 \,S^{\B} (N)  + \rho^{\B} N - \left(\rho^{\B} N\right)^2.
\end{equation}

From Lemma \ref{BM} (see Section 2) we know that $$c^{\B}_{2} (h) = \prod\limits_{b |h}\left(1-\frac 1 {b}\right) \prod\limits_{b \nmid h}\left(1-\frac 2 {b}\right).$$
Therefore, since $2\not \in\B$,
\begin{align*}
	S^{\B} (N)&=  \sum_{h=1}^{N-1} \left[(N-h) \prod\limits_{b |h}\left(1-\frac 1 {b}\right) \prod\limits_{b \nmid h}\left(1-\frac 2 {b}\right)\right] \\
	&= \prod_{b \in \B} \left(1-\frac 2 {b}\right) \cdot \sum_{h=1}^{N-1} \left[(N-h) \prod\limits_{b |h}\left(1-\frac 1 {b}\right)\left(1-\frac 2 {b}\right)^{-1}\right]\\
	&= \prod_{b \in \B} \left(1-\frac 2 {b}\right) \cdot \sum_{h=1}^{N-1} \left[(N-h) \prod\limits_{b |h}\left(\frac {b-1} {b-2}\right)\right].
\end{align*}
Introducing $C^{\B} =\prod_{b} \left(1-\frac 2 {b}\right)$, we have

\begin{align*}
	S^{\B} (N)&= C^{\B} \, \sum_{h=1}^{N-1} \left[(N-h) \prod\limits_{b |h}\left(1+ \frac 1 {b-2}\right)\right].
\end{align*}

By Lemma \ref{Perr}, this can be transformed into 
\begin{equation*}
S^{\B} (N) =  \frac {C^{\B}} {2\pi i} \int\limits_{\kappa- i\infty}^{\kappa+i\infty} F^{\B}(s) N^{s+1}\, \,\frac {d s} {s(s+1)}
\end{equation*} where $F^{\B}(s)$ is the Dirichlet series of $f(h):= \prod\limits_{b |h}\left(1+ \frac 1 {b-2}\right)$ with the abscissa of convergence $\sigma_c$ and $\kappa>\max(0,\sigma_c)$. In order to find an appropriate $\kappa$ and evaluate the integral, we need to study analytical properties of $F^{\B}(s)$. 

Let $\mathcal{Q}^{\B}$ be the set of $\B$--squarefree integers, i.e., all possible products of the elements of $\B$ without multiplicities: $\mathcal{Q}^{\B}:=\{n | \,\,n=b_1\,b_2\ldots b_s, \,  b_i \in \B\}$.

Then
\begin{align*}
F^{\B}(s) &= \sum_{h=1}^{\infty} \frac {f(h)} {h^s} =  \sum_{h=1}^{\infty} \frac 1 {h^s} \prod\limits_{b |h}\left(1+ \frac 1 {b-2}\right) =  \sum_{h=1}^{\infty} \frac 1 {h^s} \sum_{\substack{Q|h,\\Q \in \mathcal{Q}^{\B}}} \prod_{b|Q} \frac 1 {b-2} \\
&=\sum_{Q \in \mathcal{Q}^{\B}} \prod_{b|Q} \frac 1 {b^s (b-2)} \sum_{h'=1}^{\infty} \frac 1 {(h')^s} \\
&=\zeta(s)  \prod_b \left(1+\frac 1 {b^s (b-2) } \right)\\
&=: \zeta(s) G^{\B} (s).
\end{align*}

For 
$$g^{\B} (m) =
\begin{cases}
\prod_{b|m} \frac 1 {b-2}, & m \in \mathcal{Q}^{\B}\\
0, & \text{otherwise},
\end{cases}$$
we have

\begin{align*}
G^{\B} (s) &=\sum_{m=1}^{\infty} \frac {g^{\B} (m)} {m^s} \\
&=\prod_b \left(1+\frac 1 {b^s (b-2) } \right) = \zeta^{\B} (s+1)\prod_b \left(1+\frac 1 {b^s (b-2)}\right) \left(1-\frac 1 {b^{s+1}}\right) \\
&=\zeta^{\B} (s+1) \prod_b \left(1+\frac 2 {b^{s+1} (b-2)} - \frac 1 {b^{2 s +1} (b-2)} \right) \\
&=: \zeta^{\B} (s+1)  K^{\B} (s).
\end{align*}

Thus,
 \begin{equation}\label{kappa}
S^{\B} (N) =  \frac {C^{\B}} {2\pi i} \int\limits_{\kappa- i\infty}^{\kappa+i\infty} \zeta(s)  \zeta^{\B} (s+1)  K^{\B} (s) N^{s+1}\, \,\frac {d s} {s(s+1)}.
\end{equation}

From Lemma \ref{extension} we know that $\zeta^{\B} (s+1)$ can be extended meromorphically to the half-plane $\operatorname{Re} s > \beta -1$, with a simple pole at $s=\alpha-1$ and no other poles. As for $K^{\B}(s)$, it is clearly holomorphic in a region containing this half-plane, and this is enough for our purposes. Therefore, in (\ref{kappa}), one can safely take $\kappa>1$.


Now, following a method similar to the one used by Hall \cite{H}, we will deform the contour; an important subtlety here is that the shape of the contour will depend on the range of $\alpha$. We consider two cases, $0<\alpha\leq\frac 1 2$ and $\frac 1 2 <\alpha<1$. The final contours for each case are shown on \textbf{Figure~1}. We will analyze the more intricate case of $0<\alpha \leq \frac 1 2$ in detail and omit the proof for $\frac 1 2<\alpha<1$ since it is analogous and follows from our discussion.
\newline

Throughout the remaining text, we assume $0<\alpha \leq \frac 1 2$; the letters $\sigma$ and $t$ will always denote the real and imaginary parts of the complex variable $s$ respectively. \\

The Riemann zeta function can be meromorphically extended to the whole complex plane by means of the functional equation $\zeta(s) = \xi(s) \zeta(1-s)$ where $$\xi(s) = 2^s {\pi}^{s-1} \sin \frac {\pi s} 2 \,\Gamma(1-s).$$ Since
\begin{equation}\label{xi}
\xi(s) = \left(\frac {2\pi} {t} \right)^{\sigma-\frac 1 2 +it} e^{i(t+\frac {\pi} 4)} \left(1+ O\left(t^{-1}\right)\right),
\end{equation}
as $t \to \infty$, uniformly in any vertical strip, $\zeta(s) = O\left(|t|^{\frac 1 2-\sigma}\right)$ when $\sigma<0$, and the line of integration can be moved to $\operatorname{Re} s = v:=-\frac 1 2 +\delta>\alpha-1$ for some small constant $\delta>0$ that we fix once and for all. Picking up two residues, at $s=1$ and $s=0$, we obtain
\begin{align*}
S^{\B} (N) =  &\frac {N^2} 2 \,C^{\B}\,\zeta^{\B}(2)K^{\B}(1) -\frac N 2\, C^{\B} \,\zeta^{\B} (1) K^{\B} (0)  \\
+ &\frac {C^{\B}} {2\pi i} \int\limits_{v- i\infty}^{v+i\infty} \zeta(s) \zeta^{\B} (s+1) K^{\B} (s) N^{s+1}\, \,\frac {d s} {s(s+1)}.
\end{align*}
Note that $C^{\B}\,\zeta^{\B}(2)K^{\B}(1) = \left(\rho^{\B}\right)^2$ and $C^{\B} \,\zeta^{\B} (1) K^{\B} (0) = \rho^{\B}$ which gives us cancellation in (\ref{varvar}) and shows $$\Var_{\nu^{\B}} (N) = \frac {C^{\B}} {\pi i} \int\limits_{v- i\infty}^{v+i\infty} \zeta^{\B} (s+1) \zeta(s) K^{\B} (s) N^{s+1}\, \,\frac {d s} {s(s+1)}.$$

This integral is absolutely convergent. 

Staying in the halfplane $\operatorname{Re} s > \beta-1$, deform the contour (symmetrically around the $\sigma$--axis), as it is shown on \textbf{Figure 1a)}. We choose $\alpha_2 = \alpha-1+\frac 1 {\log N}$ and $T_2 = N^{1 -2\alpha+4\delta}$. Also, put $T_1 = e^{c_1 \sqrt{ \log N}}$ for some constant $c_1$ that we will pick later and $\alpha_1= \alpha-1 - \frac c {\log T_1} = \alpha-1 - \frac c {c_1 \sqrt{\log N}}$ for $c$ as in Lemma \ref{zero-free}. 
\\

We picked up a residue at $s=\alpha-1$ equal to $C\cdot N^{\alpha}$ where $C$ is independent of $N$ and equals
\begin{align*}
&\frac{2 A} {\alpha-1} \prod_{b\in \B} \left(1-\frac 2 b + \frac 2 {b^{\alpha+1}} -\frac 1 {b^{2\alpha}}\right) \xi(\alpha-1) \zeta(2-\alpha) = \\
&\frac{2^{\alpha} {\pi}^{\alpha-2} A} {\alpha-1} \prod_{b\in \B} \left(1-\frac {2} b + \frac 2 {b^{\alpha+1}} -\frac 1 {b^{2\alpha}}\right) \sin\left(\frac {\pi(\alpha-1)} 2\right) \Gamma(2-\alpha) \zeta(2-\alpha).
\end{align*}

The only job that we are left with is proving that the contribution from the contour integral is $o(N^{\alpha})$, as $N\to\infty$. Divide the upper part of the contour into five parts as shown on $\textbf{Figure 1a)}$ and consider them separately (the bound for the lower part is completely analogous).

\newpage
\begin{minipage}[c]{0.5\textwidth}
\vspace{1cm}
\begin{center}
\begin{tikzpicture}
\def\gap{0.2}
\def\br{7}
\def\lr{2}

\draw [thick](\lr*0.2, 0.8*\br) -- (0.2*\lr,\br);
\draw  [thick] (-0.1*\lr, 0.8*\br) -- (0.2*\lr,0.8*\br);
\draw [thick] (-0.35*\lr,0.5*\br)--(-0.1*\lr, 0.5*\br)--(-0.1*\lr, 0.8*\br);
\draw [thick] (-0.35*\lr, 0.5*\br)--(-0.35*\lr, -0.5*\br);
\draw [thick] (0.2*\lr,-\br)--(0.2*\lr,-0.8*\br)--(-0.1*\lr, -0.8*\br)--(-0.1*\lr, -0.5*\br) -- (-0.35*\lr,-0.5*\br);
\draw[dashed] (-0.8*\lr,-\br) -- (-0.8*\lr, \br);
\draw[dotted, thick] (0.7*\lr, 0.8*\br) --(0.2*\lr, 0.8*\br) -- (0.2*\lr, -0.8*\br);
\draw[dotted, thick] (0.7*\lr, 0.5*\br) -- (-0.1*\lr, 0.5*\br) -- (-0.1*\lr, -0.5*\br);

\draw [help lines,->] (-\lr, 0) -- (\lr,0);
\draw [help lines,->] (\lr*0.7, -1*\br) -- (\lr*0.7, 1*\br);
\node at (0.7*\lr-0.2,0.8*\br+0.2){\scriptsize$T_2$};
\node at (0.7*\lr-0.2,0.5*\br+0.2){\scriptsize$T_1$};
\node at (-0.35*\lr-0.19,-0.19){\scriptsize$\alpha_1$};
\node at (-0.1*\lr+0.22,-0.19){\scriptsize$\alpha_2$};
\node at (0.2*\lr+0.19,-0.17){\scriptsize$v$};
\node at (\lr,-0.2){$\sigma$};
\node at (\lr*0.7-0.2,\br-0.1) {$t$};
\node at (0.2*\lr-0.2,0.9*\br+0.15) {\scriptsize$I_5$};
\node at (-0.22*\lr,0.5*\br+0.2) {\scriptsize$I_2$};
\node at (-0.0*\lr,0.8*\br+0.2) {\scriptsize $I_3$};
\node at (-0.1*\lr-0.2,0.7*\br) {\scriptsize $I_4$};
\node at (-0.35*\lr-0.2,0.22*\br) {\scriptsize$I_1$};
\node at (-0.8*\lr-0.2, -0.2) {\scriptsize$\beta'$};
\end{tikzpicture}
\end{center}
\vspace{0.3cm}

\begin{center}
	\scriptsize \textbf{a) The contour in case $0 <\alpha \leq \frac 1 2$}
\end{center}

\end{minipage}
\begin{minipage}[c]{0.5\textwidth}
\vspace{1cm}
\begin{center}
\begin{tikzpicture}
\def\gap{0.2}
\def\br{7}
\def\lr{2}

\draw [thick] (0.3*\lr,\br) -- (\lr*0.3, 0.5*\br) -- (0.05*\lr, 0.5*\br) -- (0.05*\lr, -0.5*\br) -- (0.3*\lr,-0.5*\br) -- (0.3*\lr,-\br);
\draw[dashed] (-0.6*\lr,-\br) -- (-0.6*\lr, \br);
\draw[dotted, thick] (0.7*\lr, 0.5*\br) --(0.3*\lr, 0.5*\br) -- (0.3*\lr, -0.5*\br);

\draw [help lines,->] (-\lr, 0) -- (\lr,0);
\draw [help lines,->] (\lr*0.7, -1*\br) -- (\lr*0.7, 1*\br);
\node at (\lr,-0.2){$\sigma$};
\node at (\lr*0.7-0.2,\br-0.1) {$t$};
\node at (0.7*\lr-0.2,0.5*\br+0.2){\scriptsize$T_1$};
\node at (-0.6*\lr-0.2, -0.2) {\scriptsize$\beta'$};
\node at (0.05*\lr-0.19,-0.19){\scriptsize$\alpha_1$};
\node at (0.3*\lr+0.22,-0.19){\scriptsize$\alpha_2$};

\end{tikzpicture}
\end{center}
\vspace{0.3cm}

\begin{center}
	\scriptsize \textbf{b) The contour in case $\frac 1 2 <\alpha <1$}
\end{center}

\end{minipage}

\begin{center}
\small{$$\beta'=\beta-1, \,\,\alpha_1 = \alpha-1 - \frac c {c_1 \sqrt{\log N}}, \,\,\alpha_2 = \alpha-1+\frac 1 {\log N}, \,\,v=-\frac 1 2 +\delta,$$ $$T_2 = N^{1 -2\alpha+4\delta}, \,\,T_1 = e^{c_1 \sqrt{\log N}}$$}

\scriptsize{for some special choice of $c$ and $c_1$ and a small $\delta$.}
\end{center}
\vspace{1cm}
\begin{center}
\normalsize
\textbf{Fig. 1} 
\end{center}

\newpage

In our estimates we will take the liberty of using the $I_j$ notation for both the integrals and the corresponding contours.
\\

\textbf{The bound on $I_1$.} Note that the contribution to the integral coming from the part of the contour with $\operatorname{Im}s \leq t_0$ for $t_0$ as in Lemma \ref{zero-free} is negligible. Ignoring this part of the contour and exploiting the lemma which states that $|\zeta^{\B}(s+1)| \lesssim \log |\operatorname{Im}s|$ on the rest of the $I_1$, we see that
 
 \begin{align*}
 |I_1| &=\frac {C^{\B}} {\pi}\left|\int_{I_1} \xi(s) \zeta(1-s) \zeta^{\B} (s+1) K^{\B} (s) N^{s+1} \frac {ds} {s(s+1)}\right| \\
 &\lesssim \int_{t_0}^{T_1} t^{1/2-\left(\alpha-1-\frac c {\log T_1}\right)} \log t \cdot N^{\alpha - \frac c {\log T_1}}\, \frac {dt} {t^2} \\
 &\lesssim N^{\alpha} e^{-c \frac {\log N} {\log T_1}} \int_3^{T_1} t^{-\alpha-1/2}\, \log t \,dt\\
 &\lesssim N^{\alpha} e^{-c \frac {\log N} {\log T_1}} \,T_1^{1/2-\alpha+\varepsilon}
 \end{align*}
fixing arbitrary small $\varepsilon>0$. Plugging in $T = e^{c_1 \sqrt{\log N}}$, we obtain
 \begin{align*}
 |I_1| \lesssim N^{\alpha} \,e^{ \sqrt{\log N} ((\frac 1 2 -\alpha+\varepsilon) c_1-\frac c {c_1})}
 \end{align*}
 which, for $c_1< \sqrt{\frac c {\frac 1 2 -\alpha +\varepsilon}}$, yields $|I_1| = o\left(N^{\alpha}\right)$.
 
  \textbf{The bound on $I_2$.} This bound follows easily from the behavior of $\xi(s)$ on this segment and our choice of $c$. Indeed,
 \begin{align*}
 |I_2| &=\frac {C^{\B}} {\pi} \left| \int_{\alpha_1+i T_1}^{c+i T_1} \xi(s) \zeta(1-s) \zeta^{\B} (s+1) K^{\B} (s) N^{s+1} \frac {ds} {s(s+1)}\right|  \\
 & \lesssim \int_{\alpha_1}^{c} N^{1+\sigma}\,T_1^{-\sigma -\frac 3 2}\, \log T_1\, d \sigma \\
 &\lesssim N^{1+c} T_1^{-\alpha_1-\frac 3 2} = N^{\alpha+\frac 1 {\log N}} T_1^{-\alpha+ \frac c {\log T}-\frac 1 2} \\
 & \lesssim N^{\alpha} T_1^{-\alpha-1/2} =o\left(N^{\alpha}\right).
 \end{align*}
 
 \textbf{The bound on $I_3$.} This estimate is rather similar to $I_2$. Due to fast growth of $T_2$,
 \begin{align*}
 |I_3| &=\frac {C^{\B}} {\pi} \left| \int_{c+i T_2}^{v+i T_2} \xi(s) \zeta(1-s) \zeta^{\B} (s+1) K^{\B} (s) N^{s+1} \frac {ds} {s(s+1)}\right|  \\
 & \lesssim \int_{c}^{v} N^{1+\sigma}\,T_2^{-\sigma -\frac 3 2}\, \log T_2\, d \sigma \\
 &\lesssim N^{1+v} T_2^{-c-\frac 3 2} \lesssim N^{1/2+\delta} T_2^{-\alpha-1/2} \\
 & \lesssim N^{1/2+\delta} T_2^{-1/2} =o\left(N^{\alpha}\right).
 \end{align*}
 \newpage
 
 \textbf{The bounds on $I_4$ and $I_5$.} These bounds follow from the 
 \begin{proposition}\label{prop}
Denote $\alpha':=\alpha-1$. For any $c=c(N)$ with $\alpha' < c <0$ and $$T=T(N) =~\omega\left( \max\left[\left(\frac {N^{c-\alpha'}} {c-\alpha'}\right)^{\frac 1 {\frac 3 2 +c}}, \left(\frac 1 {c-\alpha'}\right)^{\frac 1 {1 + \alpha'}}\right]\right),$$ define 
 $$\mathcal{I}=\mathcal{I}_{c,T} = \int_{c+iT}^{c+i\infty} N^{s+1} \xi(s) \zeta(1-s) G^{\B} (s) \frac {ds} {s(s+1)}.$$
 Then $|\mathcal{I}| = o\left(N^{\alpha}\right)$, as $N \to \infty$.
 \end{proposition}
 
One can easily verify that our pairs $(\alpha_2, T_1) =\left(\alpha-1+\frac 1 {\log N}, e^{c_1 \sqrt{\log N}} \right)$ and $ (v, T_2) = \left(-\frac 1 2 +\delta, N^{1 -2\alpha+4\delta}\right)$ satisfy the conditions of Proposition~\ref{prop}. Therefore, once we prove the Proposition, we are done.

Recall that $$\zeta^{\B}(s+1)K^{\B}(s) = G^{\B} (s) = \prod_b \left(1+\frac 1 {b^s(b-2)}\right) = \sum\limits_{m=1}^{\infty} \frac {g^{\B}(m)} {m^s}.$$ 
 
We first prove a simple auxiliary lemma on the rate of convergence for $G^{\B} (s)$ on the real line.

\begin{lemma}\label{conv}
There exists some universal constant $C$ (depending on $\B$ only) such that for any real $c>\alpha' = \alpha-1$ and any real $x>0$, $\sum_{m>x} \frac {g^{\B} (m)} {m^c} \leq C \frac{x^{\alpha' -c}} {c-\alpha'}$.
\end{lemma}
\begin{proof} Recall $C^{\B} = \prod_b \frac {b-2} b$ and note that, for any $m \in \mathcal{Q}^{\B}$, 
\begin{equation} \label{gbound}
g^{\B}(m) = \prod_{b|m} \frac 1 {b-2} = \frac 1 m \, \prod_{b|m} \frac b {b-2}  \leq\frac 1 {C^{\B}} \cdot \frac 1 m
\end{equation}
Therefore, it is sufficient to show $$\sum\limits_{\substack{m>x, \\ m \in \mathcal{Q}^{\B}}} \frac 1 {m^{1+c}} \leq C\frac{x^{\alpha' -c}} {c-\alpha'}$$ for some constant $C$.

This follows from the fact that $\sum_{\substack{m>x, \\ m \in \mathcal{Q}^{\B}}} \frac 1 {m^{1+c}} = \int_x^{\infty} u^{-(1+c)} d N^{\B} (u)$. Recall $N^{\B} (u) = A u^{\alpha} + O\left(u^{\beta}\right)$ and perform integration by parts to obtain the result.

\end{proof}
 
\begin{proof}[Proof of Proposition~\ref{prop}]
Taking into account (\ref{xi}), we see that 
$$\mathcal{I} \lesssim  N^{c+1} \sum\limits_{m=1}^{\infty} \sum\limits_{n=1}^{\infty} \frac {g^{\B}(m)} {n^{1-c} m^{c}} \left|\mathcal{I}_{n,m}\right|$$ for $$\mathcal{I}_{n,m} = \int_{T}^{\infty} \left(\frac {2\pi e N n} {m t} \right)^{it} t^{-\frac 3 2 -c} dt.$$ 

Define $$\mathcal{J}_{n,m} (t) = \int_{T}^t \left(\frac {2\pi e N n} {m t} \right)^{iu} du.$$ Then $|\mathcal{J}_{n,m} (t)| \lesssim 1$ for $t < \frac {Nn} {m}$ and $|\mathcal{J}_{n,m} (t)| \lesssim t^{\frac 1 2}$ for all $t$ (for a proof, see \cite{H}). Therefore, since $c>-1$,
\begin{align*}
\mathcal{I}_{n,m} &= \int_{T}^{\infty} t^{-\frac 3 2 -c} d \mathcal{J}_{n,m} (t)\\
&\lesssim \int_{T}^{\infty} |\mathcal{J}_{n,m} (t)| \,t^{-\frac 5 2 -c} dt.
\end{align*}

and

$$|\mathcal{I}_{n,m}| \lesssim
\begin{cases}
T^{-\frac 3 2 - c} + \left(\frac {Nn} {m}\right)^{-1-c}, & \text{if  } T < \frac {Nn} {m};\\
T^{-1-c}, & \text{if  } T \geq \frac {Nn} {m}.
\end{cases}$$

Hence

\begin{align*}
|\mathcal{I}|  \lesssim \,\,&N^{c+1} \,T^{-\frac 3 2 - c} \sum_{\substack{n,m:\\ T < \frac {Nn} {m}}} \frac {g^{\B}(m)} {n^{1-c} m^{c}} + \sum_{\substack{n,m:\\ T < \frac {Nn} {m}}} \frac {g^{\B}(m)} {n^2 m^{-1}} \\
+\,\,&N^{c+1} \,T^{-1-c} \sum_{\substack{n,m:\\ T \geq \frac {Nn} {m}}} \frac {g^{\B}(m)} {n^{1-c} m^{c}}\\
\\
 =: \,&\mathcal{I}_{1} + \mathcal{I}_{2} +\mathcal{I}_{3}.
\end{align*}

We start with $\mathcal{I}_{1}$.
\begin{align*}
\mathcal{I}_{1} &= N^{c+1} \,T^{-\frac 3 2 - c} \sum_{m=1}^{\infty} \frac {g^{\B}(m)} {m^{c}} \sum_{n>\operatorname{max} \left(1, \frac {mT} {N}\right)} \frac 1 {n^{1-c}} \\
&\lesssim N^{c+1} \,T^{-\frac 3 2 - c} \left\{\sum_{m<\frac N {T}} \frac {g^{\B}(m)} {m^{c}} + \sum_{m\geq \frac N {T}} \frac {g^{\B}(m)} {m^{c}} \left(\frac {m T} {N} \right)^{c}\right\} \\
&\lesssim N^{c+1} \,T^{-\frac 3 2 - c} \left\{\sum_{m=1}^{\infty} \frac {g^{\B}(m)} {m^{c}} + \left(\frac {T} N\right)^{c} \sum_{m\geq \frac N {T}} g^{\B}(m)\right\}.
\end{align*}
By Lemma \ref{conv}, we obtain
\begin{align*}
\mathcal{I}_1& \lesssim \frac 1 {c-\alpha'} N^{c+1} \,T^{-\frac 3 2 - c} \left\{1 + \left(\frac {T} N\right)^{c} \left(\frac N T \right)^{\alpha'}\right\} \\
&\lesssim \frac 1 {c-\alpha'} \left\{N^{c+1} \,T^{-\frac 3 2 - c} + N^{\alpha} T^{-\frac 3 2 -\alpha'}\right\}.
\end{align*} 
Since we are given $T=\omega\left( \max\left[\left(\frac {N^{c-\alpha'}} {c-\alpha'}\right)^{\frac 1 {\frac 3 2 +c}}, \left(\frac 1 {c-\alpha'}\right)^{\frac 1 {1 + \alpha'}}\right]\right)$ and, for the exponents, $\frac 3 2 +\alpha' > 1+\alpha'$, we see that $\mathcal{I}_1 = o(N^{\alpha})$.

Next, consider
\begin{align*}
\mathcal{I}_{2}&=\sum_{\substack{n,m:\\  m<\frac {Nn} {T}}} \frac {m g^{\B}(m)} {n^2} \lesssim \sum_{n} \frac{\mathcal{N}^{\B} \left(\frac {Nn} {T}\right)} {n^2} \lesssim \left(\frac {N} {T}\right)^{\alpha} \sum_n n^{\alpha-2} =o(N^{\alpha}).
\end{align*}

As for $\mathcal{I}_{3}$, again, by Lemma \ref{conv}, and the fact that $T^{\alpha}=T^{1+\alpha'}=\omega\left(\frac 1 {c(N)-\alpha'}\right)$,
\begin{align*}
\mathcal{I}_{3} &= N^{c+1} \,T^{-1-c} \sum_n \frac 1{n^{1-c}} \sum_{m \geq \frac {Nn} {T}} \frac {g^{\B}(m)} { m^{c}} \\
&\lesssim N^{c+1} \,T^{-1-c} \sum_n \frac 1{n^{1-c}} \cdot \frac 1 {c-\alpha'} \left(\frac {Nn} {T}\right)^{\alpha'-c} \\
&= N^{\alpha} \frac {T^{-\alpha}} {c-\alpha'} \sum_n n^{\alpha-2} = o(N^{\alpha}).
\end{align*}
Gathering all the estimates, we see $\mathcal{I} = o(N^{\alpha})$.
\end{proof}

\section{Proof of Theorem \ref{okvar}}\label{th2proof}
\begin{proof}
It can be seen from the proof of Theorem \ref{bvar} that, in order to estimate the variance, it is sufficient to use the formulae for the second correlations. Indeed, in complete analogy to (\ref{varvar}), it is easy to show that 

\begin{align}\label{varok2}
\Var_{\ok}^{(k)} (N) &=  \sum_{0\leq ||a_1||, ||a_2|| \leq N} c_{\ok, 2}^{(k)} (a_2-a_1) - \left(\rho_{\ok}^{(k)}\,\# B_N \right)^2\\
&=: M_{2,\ok}^{(k)} - \left(\rho_{\ok}^{(k)}\,\# B_N\right)^2.
\end{align}
Here, we will need the statement proven by Cellarosi and Vinogradov in \cite{CV}:
\begin{lemma}[Correlation functions of $k$--free integers in a number field] For every $r\geq 1$ and every $a_1, \cdots a_r \in \ok$,
\begin{equation}
c^{(k)}_{\ok, r+1} (a_1, \ldots a_r) = \prod_{\p} \left(1-\frac{D(\p^k |0, a_1, \ldots a_r) } {N(\p^k)} \right).
\end{equation}
where $$D(\p^k |0, a_1, \ldots a_r)=\#\{b \,\mod \,\p^k\, |\, b \equiv a_i\,\mod\, \p^k \text{ for some } i=0\ldots r, \, a_0 =0\}.$$
\end{lemma}
Plugging this formula into (\ref{varok2}) for $r=1$, we obtain:
\begin{align*}
M_{2,\ok}^{(k)} (N) &=  \sum_{0\leq ||a_1||, ||a_2|| \leq N} \prod_{\p} \left(1-\frac{D(\p^k |0, a_2-a_1) } {N(\p^k)} \right) \\
&= \sum_{0\leq ||a_1||, ||a_2|| \leq N} \prod_{\substack{\p,\\ a_2 \equiv a_1 \operatorname{mod} \p^k}} \left(1-\frac 1 {N(\p^k)}\right) \prod_{\substack{\p,\\ a_2 \not\equiv a_1 \operatorname{mod} \p^k}} \left(1-\frac 2 {N(\p^k)}\right) \\
& = C_{\ok}^{(k)} \sum_{0\leq ||a_1||, ||a_2|| \leq N} \prod_{\substack{\p,\\ a_2 \equiv a_1 \operatorname{mod} \p^k}} \left(1+\frac 1 {N(\p^k)-2}\right)
\end{align*}
where $C_{\ok}^{(k)} = \prod_{\p} \left(1-\frac 2 {N(\p^k)}\right)$.

Let $\mathcal{Q}_{\ok}^{(k)}$ be the set of ideals of the form $\mathfrak{d} = \p^k_1 \ldots \p^k_s$ for some different prime $\p_i, i=1...s$. For such a $\mathfrak{d} \in \mathcal{Q}_{\ok}^{(k)} $, we define $$g_{\ok}^{(k)} (\mathfrak{d}) = \prod_{\p_i, i =1 \ldots s} \frac 1 {N(\p_i^k) -2}.$$ Then, clearly,
\begin{align*}
M_{2,\ok}^{(k)} (N)&= C_{\ok}^{(k)} \sum_{\mathfrak{d} \in \mathcal{Q}_{\ok}^{(k)}} g_{\ok}^{(k)} (\mathfrak{d})  \sum_{\substack{0\leq ||a_1||, ||a_2|| \leq N, \\ a_2-a_1 \in \mathfrak{d}}} 1.
\end{align*}

Here we need another standard fact which we give without a proof (it follows from the lattice description of ideals of $\ok$, see \cite{CV}, \cite{N}):
\begin{lemma}[Counting the elements of an ideal in an $L^1$--ball]\label{ballcount}
Let $a\in \ok$, $\mathfrak{d}$ be an ideal in $\ok$ and
$$T(x) = \# \{n \in \ok: ||n||\leq x\text{ and }n + a \equiv 0 \, \operatorname{mod} \mathfrak{d}\}.$$
Then
$$\left|T(x)-\frac {\# B_x} {N(\mathfrak{d})}\right| = O\left(\frac {\operatorname{diam} \mathfrak{d}} {N(\mathfrak{d})}\, x^{d-1}\right).$$
\end{lemma}

From this Lemma we deduce that
\begin{align*}
M_{2,\ok}^{(k)} (N) &=  C_{\ok}^{(k)} \sum_{\mathfrak{d} \in \mathcal{Q}_{\ok}^{(k)}} g_{\ok}^{(k)} (\mathfrak{d}) \left[\frac {\left(\# B_N\right)^2} {N(\mathfrak{d})} + O \left( \frac {\operatorname{diam} \mathfrak{d}} {N(\mathfrak{d})}\, N^{2d-1} \right)\right].
\end{align*}

Since the error term can be bounded by $$\lesssim \sum_{\mathfrak{d} \in \mathcal{Q}_{\ok}^{(k)}} \frac{ g_{\ok}^{(k)} \operatorname{diam} \mathfrak{d}} {N(\mathfrak{d})}\, N^{2d-1} = O_{K,k} \left(N^{2d-1}\right),$$ we are only left with showing that $$C_{\ok}^{(k)} \sum_{\mathfrak{d} \in \mathcal{Q}_{\ok}^{(k)}} \frac {g_{\ok}^{(k)} (\mathfrak{d})} {N(\mathfrak{d})} = \left(\rho_{\ok}^{(k)}\right)^2.$$ Indeed,
\begin{align*}
C_{\ok}^{(k)} \sum_{\mathfrak{d} \in \mathcal{Q}_{\ok}^{(k)}} \frac {g_{\ok}^{(k)} (\mathfrak{d})} {N(\mathfrak{d})} & = \prod_{\p} \frac {N(\p^k)-2} {N(\p^k)} \prod_{\p} \left (1 +\frac 1 {N(\p^k) (N(\p^k)-2)} \right) \\
&= \prod_{\p} \frac {N(\p^k)-2} {N(\p^k)} \prod_{\p} \frac {\left(N(\p^k) - 1\right)^2} {N(\p^k) (N(\p^k)-2)} \\
&= \prod_{\p} \frac {\left(N(\p^k) - 1\right)^2} {(N\left(\p^k)\right)^2} =  \left(\rho_{\ok}^{(k)}\right)^2.
\end{align*}
Therefore,
\begin{equation*}
\Var_{\ok}^{(k)} (N) = M_{2,\ok}^{(k)} - \left(\rho_{\ok}^{(k)}\,\# B_N\right)^2 = O_{K, k}\left(N^{2d-1}\right).
\end{equation*}
The theorem is proven.
\end{proof}

\end{document}